\documentclass[reqno,11pt]{amsart}
\usepackage[utf8]{inputenc}

\usepackage{amsmath}
\usepackage{mathtools}
\usepackage{amsfonts}
\usepackage{amssymb}
\usepackage{amsmath}
\usepackage{amsthm}
\usepackage{amssymb}
\usepackage{latexsym}
\usepackage{verbatim}
\usepackage[dvipsnames]{xcolor}
\usepackage{tikz-cd}
\usepackage{graphicx}
\usepackage{chngcntr}
\usepackage{apptools}
\usepackage[capitalize]{cleveref}

\AtAppendix{\counterwithin{lemma}{section}}

\theoremstyle{definition}
\newtheorem{definition}{Definition}[section]
\theoremstyle{definition}
\newtheorem{proposition}{Proposition}[section]
\newtheorem*{proposition*}{Proposition}

\newtheorem{lemma}{Lemma}[section]
\newtheorem{remark}{Remark}[section]
\newtheorem{assumption}{Assumption}[section]
\theoremstyle{corollary}

\Crefname{assumption}{Assumption}{Assumptions}

\newcommand{\dtt}{ \partial_{tt}}

\newcommand{\md}{\mathcal{D}}
\newcommand{\mcu}{\mathcal{U}}
\newcommand{\mcv}{\mathcal{V}}
\newcommand{\mcw}{\mathcal{W}}
\newcommand{\lb}{\bigtriangleup}
\newcommand{\intm}{\int_M}
\newcommand{\intb}{\int_{\partial M}}
\newcommand{\ol}[1]{\overline{#1}} 
\newcommand{\mx}{\mathcal{X}}

\title[The nonlinear wave equation with Wentzell  conditions]{The nonlinear wave equation with nonlinear Wentzell boundary conditions on time-dependent compact Riemannian manifolds.}

\author[A.\ Marta]{Alessio Marta}
\address{Dipartimento di Scienze del Sistema Nervoso e del Comportamento\\ Universit{\`a} degli Studi di Pavia\\ Via Bassi, 21 --  I-27100 Pavia \\ Italy}
\email{alessio.marta@unipv.it}

\begin{document}

\maketitle
\begin{abstract}
We prove a local well-posedness result for an evolution problem consisting of a semilinear wave equation with subcritical nonlinearities posed on a time-dependent compact Riemannian manifold and supplied with a nonlinear dynamical boundary condition of Wentzell type.
\end{abstract}

\section{Introduction and main results}
Let $(M,g_t)$ be a smooth, oriented, compact Riemannian manifold  with boundary $\partial M$ and suppose that $g_t$ is a time-dependent Riemannian metric defined for $t \in [0,T]$, $T>0$, satisfying the following hypothesis.
\begin{assumption}\label{hypothesis:smooth_metric}
The family of Riemannian metrics $\{g_t\}_{t \in [0,T]}$ depends smoothly on $t$. 
\end{assumption}
Let $\gamma_- : H^{s}(M) \rightarrow H^{s-1/2}(\partial M)$, $s \in \mathbb{R}$, be the Lion-Magenes trace \cite{LionMagenes1,LionMagenes2}. Notice that the definition of this trace is time-independent because $M$ is compact and therefore the Sobolev space $H^s(M)$ does not depend on the choice of the Riemannian metric. Let $\nu$ be the outward-pointing unit normal of $\partial M$ and define the Neumann trace map  
$$\gamma_+ := \gamma_- \circ \partial_\nu : H^{s}(M) \rightarrow H^{s-3/2}(\partial M).$$ For $t \in [0,T]$, $t \in \mathbb{R}_+$ we consider the semilinear problem
\begin{equation}\label{eq:non_linear_problem}
\begin{cases}
\dtt u - \lb (t) u - m(t,p) u  = \mathcal{N}(t,p,u)\ \textit{in} \ M^\circ \\
\dtt v  - \lb_b (t) v - m_b(t,q) v - \rho = \mathcal{N}_b(t,q,v) \ \textit{on} \ \partial M \\
v = \gamma_- u, \quad \rho = - \gamma_+ u
\end{cases}
\end{equation}
where $\lb (t)$ and $\lb_b (t)$ are the Laplace-Beltrami operators built out of the metrics $g_t$ and $h_t=g_t|_{\partial M}$ respectively, $p \in M$, $q \in \partial M$ and $\mathcal{N}$, $\mathcal{N}_b$ are nonlinear terms. The functions $m$ and $m_b$ satisfy the following hypotheses.
\begin{assumption}\label{hyp:mass}
$m \in \mathcal{C}^1([0,T] \times M)$ and $m_b \in \mathcal{C}^1 ([0,T] \times \partial M)$ are non-positive functions.
\end{assumption}
\begin{remark}
This assumption is needed to prove a well-posedness result for the non-autonomous linear equation without recurring to perturbative methods.
\end{remark}
Our assumptions on the nonlinearities are the following.
\begin{assumption}\label{hyp:nonlinearities_1}
The nonlinear terms $\mathcal{N}$ and $\mathcal{N}_b$ seen as functions from $([0,T] \times M) \times \mathbb{R}$ and $([0,T] \times \partial M) \times \mathbb{R}$ to $\mathbb{C}$ are Carath\'eodory functions satisfying the following conditions:
\begin{itemize}
\item[a)] For every $u,v \in \mathbb{R}$ and $(t,p) \in [0,T] \times M$ and $(t,q) \in [0,T] \times \partial M$ there are $C,C_b \geq 0$ and $\alpha,\beta \geq 1$ such that 
\begin{equation}\label{eq:non_linearities_form}
\begin{split}
|\mathcal{N}(t,p,u)| \leq C ( 1 + |u|^{\alpha} )\\
|\mathcal{N}_b(t,q,v)| \leq C_b ( 1 + |v|^{\beta})
\end{split}
\end{equation}
\item[b)]  For every $u_1,u_2,v_1,v_2 \in \mathbb{R}$ and $(t,p) \in [0,T] \times M$ and $(t,q) \in [0,T] \times \partial M$ there are $K,K_b \geq 0$ and $\alpha,\beta \geq 1$ such that 
\begin{equation*}
|\mathcal{N}(t,p,u_1)-\mathcal{N}(t,p,u_2)| \leq K |u_1-u_2| \left( 1 + |u_1|^{\alpha-1} + |u_2|^{\alpha-1} \right)
\end{equation*}
\begin{equation*}
|\mathcal{N}_b(t,p,v_1)-\mathcal{N}_b(t,p,v_2)| \leq K_b |v_1-v_2| \left( 1 + |v_1|^{\beta-1} + |v_2|^{\beta-1} \right)
\end{equation*}
\end{itemize}
\end{assumption}
\begin{remark}
The previous hypothesis is satisfied by power type non-linearities of the form
\begin{equation*}
\begin{split}
\mathcal{N}(t,x,y) & = P(t,x) |y|^{\alpha-1}y\\
\mathcal{N}_b(t,x,z) & = P_b(t,x) |z|^{\beta-1}z
\end{split}
\end{equation*}
with $P \in \mathcal{C}^0([0,T]\times M)$, $P_b \in \mathcal{C}^0([0,T]\times \partial M)$ and $\alpha,\beta$ real positive numbers.
\end{remark}
Let $\mathsf{C_{M^\circ}}$ and $\mathsf{C_{\partial M}}$ be the critical exponents of the Sobolev embeddings for $H^1(M)$ and $H^1(\partial M)$, namely
\begin{equation}\label{eq:critical_exponents}
\mathsf{C_{M^\circ}} =
\begin{cases}
\dfrac{2n}{n-2} \ & \textit{if} \ n \geq 3\\
\infty \ & \textit{if} \ n = 2
\end{cases}
\quad
\mathsf{C_{\partial M}} =
\begin{cases}
\dfrac{2n-2}{n-3} \ & \textit{if} \ n \geq 4\\
\infty \ & \textit{if} \ n = 2,3
\end{cases}
\end{equation}
Then we assume the following hypothesis on the exponents $\alpha,\beta$ appearing in the nonlinear terms.
\begin{assumption}\label{hyp:nonlinearities_2}
The exponents $\alpha,\beta$ in \Cref{eq:non_linearities_form} satisfy
$$1 \leq \alpha \leq \dfrac{\mathsf{C_{M^\circ}}}{2}$$
$$1 \leq \beta \leq \dfrac{\mathsf{C_{\partial M}}}{2}$$
In the case $\mathsf{C_{M^\circ}}=\infty$ or $\mathsf{C_{\partial M}}=\infty$ we understand the above inequalities as $\alpha \geq 1$  and $\beta \geq 1$.
\end{assumption}
\begin{remark}
Note that our hypotheses in the case $n=3$ are encompassing a cubic interaction in the interior, a relevant case in quantum field theory \cite{GlimmJaffe,Peskin}.
\end{remark}

The problem in \Cref{eq:non_linear_problem} can be seen as a system of PDEs describing a physical system in which a field in the interior and one on the boundary have different dynamics and the two theories are coupled by means of the term $\rho$, with the field in the interior being a source for the boundary one. These kind of PDEs appear in quantum field theory on curved spacetimes \cite{Zahn17,DFJ18,DDF18,DM20,JW20,JW21,DJM22} -- in general with a time-dependent metric -- and in physical models where the momentum of the boundary cannot be neglected, for example in acoustic flow problems \cite{LL98,Li21}. Often in problems from classical physics $-m$ and $-m_b$ are both constant and represent the bulk and the boundary mass respectively. However, in quantum field theory on curved spacetimes these two linear terms are in general non-constant \cite{DFJ18,DDF18}.\\

Well-posedness results, resolvent estimates, asymptotics and stability of the solutions of similar problems have been treated by several authors with different techniques and regularity hypotheses \cite{Gal04,Coclite14,Vitillaro16,DDF18,Csobo2019,LiChan21,Vanspranghe20,Mehdi22,Vanspranghe22} in the case of time independent Laplace-Beltrami operators. The 1-D non-autonomous wave equation with Wentzell boundary conditions has been studied in \cite{Favini05}. In quantum field theory the linear version of problem \eqref{eq:non_linear_problem} was studied in the anti-de Sitter spacetime in \cite{DFJ18} and in Minkowski spacetime in \cite{Zahn17}.\\

Our aim is to establish a local well-posedness result for the non-au\-tonomous problem given in \Cref{eq:non_linear_problem}. We accomplish this goal making use of a technique already employed to treat the autonomous problem in \cite{Vitillaro16,Vanspranghe20,Vanspranghe22}, namely  recasting \Cref{eq:non_linear_problem} as a first order evolution equation of the form
\begin{equation}\label{eq:evolution_formulation_simplified}
\dfrac{dX}{dt}(t)  = A(t)X(t) + \mathcal{N}(t,X)
\end{equation}
with $X=(u,\partial_t u,v,\partial_t v)$ in a certain subspace $\mx$ of $L^2(M^\circ) \times L^2(\partial M) \times H^1(M^\circ) \times H^1(\partial M)$ that we shall introduce in \Cref{subsec:evolution_formulation}. In this expression $A$ is an operator expressing the linear part of the equation and $\mathcal{N}$ is a term containing the nonlinearities. We begin studying the linearized version of \Cref{eq:non_linear_problem} with a time-independent metric, proving that for every $s \in [0,T]$ the operator $A(s)$ is the generator of a $\mathcal{C}^0$ contraction semigroup. This result allows us to proceed with the study of the non-autonomous case, considering a time-dependent metric. First, we obtain a global well-posedness result for the non-autonomous linear problem, which can be roughly stated as follows.
\begin{proposition*}[\Cref{prop:low_regularity_non_auto_well_posedness}]
Let $F \in \mathcal{C}^1([0,T],\mx)$ and suppose that $X(0) \in \md (A)$. Then the linear non-autonomous problem
\begin{equation*}
\begin{cases}
\dot{X} = A(t) X(t) + F(t), \quad t \in [0,T]\\
X(0) = X_0
\end{cases}
\end{equation*}
admits a unique solution $X(t) \in \mathcal{C}^1([0,T],\mx) \cap \mathcal{C}^0([0,T],\mathcal{D}(A))$. Furthermore, if $X_0 \in \mx$ the problem admits a unique mild solution $X \in \mathcal{C}^0([0,T],\mx)$ such that
\begin{equation*}
\| X(t) \|_\mx \leq \|  X(0) \|_\mx + t \max_{s \in [0,T]} \| F(s) \|_\mx. 
\end{equation*}
\end{proposition*}
Then we focus on the non-linear case. Making use of a fixed-point argument we establish the following existence and uniqueness for the non-autonomous nonlinear problem.
\begin{proposition*}[\Cref{prop:nonlinear_well_posedness_conditional}]
Assume that Hypotheses \eqref{hypothesis:smooth_metric}, \eqref{hyp:nonlinearities_1} and \eqref{hyp:nonlinearities_2} hold true. Let $\rho > 0$. Then there is $\tau = \tau(\rho)$ such that for every $X_0 \in \mx$ with $\| X_0 \|_\mx \leq \rho$, the semilinear problem as per \Cref{eq:evolution_formulation_simplified} admits a unique mild solution $X \in C^0([0,\tau],\mx)$.
\end{proposition*}
From this conditional existence and uniqueness proposition we are then able to prove the following local well-posedness result.
\begin{proposition*}[\Cref{prop:nonlinear_well_posedness_unconditional}]
Assume Hypotheses \eqref{hypothesis:smooth_metric}, \eqref{hyp:nonlinearities_1} and \eqref{hyp:nonlinearities_2}. Then the following assertions hold true:
\begin{itemize}
\item[a)] For each $u_0 \in X$ there is a maximal mild solution $X(t) \in \mathcal{C}^0(I,X)$ with $I$ either $[0,T]$ or $[0,t^+(X_0))$, where $t^+(X_0) \in [\tau,T]$.
\item[b)] If $t^+(X_0) < T$, then $\lim_{t\rightarrow t^+(X_0)^-} \|X(t) \|=+\infty$.
\item[c)] For any $t^\star \in (0,t^+(X_0))$ there is a radius $\rho=\rho(X_0,t^\star)$ such that the map
$$ \overline{B(X_0,\rho)} \rightarrow \mathcal{C}^0([0,b],X), \quad X_0 \mapsto X(t)$$
is Lipschitz continuous.
\end{itemize}
\end{proposition*}

The paper is organized as follows. In \Cref{sec:analytic_preliminaries} we introduce some basic facts on semigroups and evolution equations. In \Cref{sec:linear_boundary_value_problem} we set the framework we use in the rest of the paper, writing \Cref{eq:non_linear_problem} as a first order evolution equation and introducing all the relevant functional spaces. \Cref{sec:time_independent_metric} is devoted to the study of the linearized problem in the case of a time-independent metric. In \Cref{sec:time_dependent_metric} we treat the case of a time-dependent metric. First we consider the linearized non-autonomous problem, showing that it is hyperbolic in the sense of evolution semigroups, then we finally focus on the nonlinear problem, obtaining the sought local well-posedness result.

\newpage

\section{Analytic preliminaries}\label{sec:analytic_preliminaries}
In this section we present some basic facts on semigroups and evolution equations. The main goal of this section is to give the precise definitions and the theorems we will employ in the remainder of the work, since in the vast literature on autonomous and non-autonomous evolution equations some definitions are not understood in the same sense. 
\subsection{Autonomous evolution problems}
\subsubsection{Strongly continous semigroup}
Let $\mx$ be a Banach space and consider a linear operator $A : D(A) \subset \mx \rightarrow \mx$. Consider the initial value problem
\begin{equation}\label{eq:acp}
\begin{cases}
\dot{X} = A X(t), \quad t \geq 0\\
X(0) = X_0 
\end{cases}
\end{equation}
for $X_0 \in \mx$.
\begin{definition}
A function $X : \mathbb{R}_0^+ \rightarrow \mx$ is called a classical solution to \Cref{eq:acp} if $X \in \mathcal{C}^1(\mathbb{R}_0^+;\mx)$, $X(t) \in D(A)$ and \Cref{eq:acp} is satisfied for every $t \geq 0$.
\end{definition}
The connection between one parameter semigroups and evolution equations is given by the following proposition.
\begin{proposition}[\cite{KlausRainer99}, Proposition 6.2]
Let $(A,D(A))$ be the generator of a strongly continuous semigroup $(T(t))_{t \geq 0}$. Then, for every $X_0 \in D(A)$, the function $X : t \mapsto X(t):=T(t)X_0$ is the unique classical solution of \Cref{eq:acp}.
\end{proposition}
A classical solution exists if and only if the initial value $X_0$ belongs to $D(A)$. However, the integral formulation of \Cref{eq:acp} may have sense also for initial data $X_0$ belonging to the Banach space $\mx$.
\begin{definition}
A continuous function $X : \mathbb{R}^+ \rightarrow \mx$ is called a mild solution to the Cauchy problem in \Cref{eq:acp} if $\int_0^t X(s) ds \in D(A)$ for every $t > 0$ and if the conditions
$$ X(t) = A\int_0^t X(s) ds  + X_0 $$
holds true $\forall t > 0$.
\end{definition}
\begin{remark}
A classical solution of the problem is also a mild one.
\end{remark}
As for classical solutions if $A$ is the generator of a strongly continuous semigroup $T(t)$, we can build the mild solutions of the evolution problems by means of $T(t)$.
\begin{proposition}[\cite{KlausRainer99}, Proposition 6.4]
Let $(A,D(A))$ be the generator of a strongly continuous semigroup $(T(t))_{t \geq 0}$. Then, for every $X \in \mx$, the orbit map $X : t \mapsto X(t):=T(t)X_0$ is the unique mild solution of \Cref{eq:acp}.
\end{proposition}
The generalization to inhomogeneous problems is straightforward \cite{Sch20}. 
\begin{definition}
Let $A$ generate a strongly continuous semigroup $(T(t))_{t \geq 0}$ and let $F: [0,T] \rightarrow \mx$ such that $\int_0^\delta \|F(s)\|_\mx ds \leq \infty$ for every $\delta \in (0,T]$. A classical solution $X(t)$ of the inhomogenous problem
\begin{equation}\label{eq:inhomogeneous_problem}
\begin{cases}
\dot{X} = A X(t) + F(t), \quad t \in [0,T]\\
X(0) = X_0
\end{cases}
\end{equation}
is a function $X(t)\in C^1((0,T],\mx) \cap C^0([0,T],\mx)$, given by Duhamel's formula
\begin{equation}\label{eq:acpsol_source}
X(t) = T(t)X_0 + \int_0^t T(t-s)F(s) ds
\end{equation}
\end{definition}
\begin{remark}
In some works a classical solution $X(t)\in C^1([0,T], \mathcal{D}(A))$ is called strict\footnote{In literature there are slightly different definitions of classical and strict solutions in the case of inhomogeneous and nonlinear problems. Here we stick to the defintions given in \cite{Sch20}}.
\end{remark}
To guarantee the existence and the uniqueness of the aforementioned solution, we need some additional hypotheses of the source term $F$.
\begin{proposition}[\cite{Sch20}, Theorem 2.9]\label{prop:inhomogeneous_evolution_well_posedness_1}
Let $(A,D(A))$ be the generator of a strongly continuous semigroup $(T(t))_{t \geq 0}$. Let $X_0 \in \mathcal{D}(A)$. Assume that either $F \in \mathcal{C}^1([0,T],\mx)$ or that $F \in \mathcal{C}^0([0,T],\mathcal{D}(A))$. Then the function $X$ given by \Cref{eq:acpsol_source} is the unique classical solution of the inhomogeneous problem \eqref{eq:inhomogeneous_problem} on $[0,T]$. Furthermore, if the initial datum $X_0 \in \mx$ the solution $u \in C^0([0,T],\mx)$ given by \Cref{eq:acpsol_source} is the unique mild solution of the problem. 
\end{proposition}
Assuming additional regularity on the source term, we can improve both the spatial and the temporal regularity of the solution.
\begin{proposition}[\cite{Can18}, Theorems 16,17]\label{prop:inhomogeneous_evolution_well_posedness_2}
Let $X_0 \in \mathcal{D}(A)$ and assume that $F$ is either in $L^2([0,T],\mx) \cap \mathcal{C}^0((0,T],\mx)$ or $H^1([0,T],\mx)$. Then the mild solution $X$ of \Cref{prop:inhomogeneous_evolution_well_posedness_1} is a strict solution, namely $X \in \mathcal{C}^1([0,T],\mx) \cap \mathcal{C}^0([0,T],\mathcal{D}(A))$.
\end{proposition}

\subsection{Non-autonomous evolution problems}\label{sec:Non-autonomous_evolution_problems}
We now turn our attention of non-autonomous problems, following \cite{Scha02Evo,Kato53}.
First, we focus on non-autonomous homogeneous problems of the form
\begin{equation}\label{eq:non_autonomous_homogeneous}
\begin{cases}
\dot{X} = A(t)X(t) + F(t)\\
X(0) = X_0, \quad t \in [0,T]
\end{cases}
\end{equation}
assuming that $\mathcal{D}(A(t))= \mathcal{D}(A(0))$ for every $t \in [0,T]$, $T \in \mathbb{R}$, is dense in $\mx$. Since the domain of $A(t)$ does not depend on time, we simply write $\mathcal{D}(A)$ to denote the domain of $A$ at a certain time. Before stating a well-posedness result, we need to introduce the notion of evolution family.
\begin{definition}
Let $J \subseteq \mathbb{R}$ be a closed interval. A family of bounded operators $U(t,s)_{t \geq s}: \mx \rightarrow \mx$ with $s,t \in J$ is an evolution family if the following conditions are satisfied:
\begin{itemize}
\item[a)] $U(t,r)U(r,s) = U(t,s)$ for every $t\geq r \geq s$, $t,r,s \in J$.
\item[b)] For every $s \in J$ $U(s,s) = \mathbb{I}$.
\item[c)] The map $(t,s) \mapsto U(t,s)$  is strongly continuous.
\end{itemize}
\end{definition}
We say that $U(t,0)$ solves a well-posed non autonomous problem if given an initial datum $X_0$, then we have that
\begin{equation*}
X(t) = U(t,s)X_0 + \int_0^t U(t,\tau)f(\tau) d\tau
\end{equation*} 
for every $t \in [0,T]$. As for autonomous equation we say that the solution is strict if $X(t) \in \mathcal{C}^1([0,T],\mathcal{D}(A))$ and mild if $X(t) \in \mathcal{C}^0((0,T],\mx)$. In the case the domain of the operators $A(t)$ is constant in time -- in this case we simply write $\mathcal{D}(A)$ instead of $\mathcal{D}(A(t))$ -- the following proposition holds true \cite{Kato53}.
\begin{proposition}
\label{prop:well_posedness_hyperbolic_problems}
Assume that the following conditions hold true:
\begin{itemize}
\item[C1)] For every $t \in [0,T]$ the operator $A(t)$ is the generator of a $\mathcal{C}^0$ semigroup.
\item[C2)] The operator $B(s,t)=[\mathbb{I}+A(t)] \cdot [\mathbb{I}+A(s)]^{-1}$ is a uniformly bounded operator for every $t,s \in [0,T]$.
\item[C3)] $B(t,s)$ is of bounded variation in $t$, namely for evert $s$ there is $N \geq 0$ such that for every partition $0=t_0<t_1<\ldots<t_n=T$ of the interval $[0,T]$ we have:
\begin{equation*}
\sum_{j=1}^n \| B(t_j,s)-B(t_{j-1},s) \|_{\mathcal{B}(\mx)} \leq N.
\end{equation*}
\item[C4)] $B(t,s)$ is weakly continuous in $t$ at least for some $s$.
\item[C5)] $B(t,s)$ is weakly differentiable in $t$ and the operator $\partial B(t,s) /\partial t$ is strongly continuous in $t$ at least for some $s$.
\end{itemize}
Then there is an evolution family $U(t,s)$ defined $0 \leq s \leq t \leq T$ such that $U(t,s) \mathcal{D}(A) \subset \mathcal{D}(A)$ and the problem in \Cref{eq:non_autonomous_homogeneous} admits a unique strict solution $X(t) \in \mathcal{C}^1([0,T],\mx) \cap \mathcal{C}^0([0,T],\mathcal{D}(A))$ for every $X_0 \in \mathcal{D}(A)$ and $F(t) \in \mathcal{C}^0([0,T],\mathcal{D}(A))$. Furthermore, if $X_0 \in \mx$ and $F(t) \in \mathcal{C}^0([0,T],\mx)$, \Cref{eq:non_autonomous_homogeneous} admits a unique mild solution $X(t) \in \mathcal{C}^0([0,T],\mx)$.
\end{proposition}
\begin{remark}
The first requirement of condition C2) is automatically satisfied when the domain of $A(t)$ is independent of $t$ by \cite[Lemma 2]{Kato53}
\end{remark}

The interested reader can find other details on non-autonomous evolution equations in \cite{Kato53,Kato61,Tanabe79,Tanabe97,Pazy83,Fattorini83,Amann95,Lunardi95,NN02,Scha02Evo}.

\section{The linear boundary value problem}\label{sec:linear_boundary_value_problem}

\subsection{The linear problem}

We start studying the linear problem
\begin{equation}\label{eq:linear_problem}
\begin{cases}
  \ \dtt u - \lb u - m u = F \ \textit{in} \ M^\circ \\
  \ \dtt v - \lb_b v - m_b u - \rho = G \ \textit{on} \ \partial M \\
v = \gamma_-(u), \quad \rho = - \gamma_+ u
\end{cases}
\end{equation}
with $F \in \mathcal{C}^1([0,T],L^2(M^\circ))$  and 
$G \in \mathcal{C}^1([0,T],L^2(\partial M))$.
\begin{remark}
Notice that we can consider different sources for the interior and the boundary. When we will treat the nonlinear boundary value problem this will allow us to consider different nonlinearities for the interior and for the boundary equations, a situation of interest, e.g., for the physical models realizing the holographic principle \cite{Witten98,Bousso02} in which the equation governing the dynamics of the field on the boundary has a nonlinear term, while the one in the interior has not.
\end{remark}

\subsection{Evolution equation formulation of the problem}\label{subsec:evolution_formulation}
Following \cite{Vitillaro16,Vanspranghe20,Vanspranghe22} we introduce the space $\mx = \mcv \times \mcu$, where 
\begin{equation}
\mcu = L^2(M^\circ) \times L^2(\partial M)
\end{equation}
and
\begin{equation}
\mcv = \left\{ (u,v) \in  H^1(M^\circ) \times H^1(\partial M) \ | \ \gamma_- u = v \right\}
\end{equation}
\begin{remark}
Since $M$ is a compact manifold, we do not need to specify which metric $g_t$ we are using in the definition of $\mcu$ and $\mcv$, as every choice gives rise to the same Sobolev space. In the following, however, the choice of the metric $g_{\hat{t}}$ at time $\hat{t}$ is understood. For example the $L^2$ norm of the function $u(\hat{t}) : M \rightarrow \mathbb{C}$ is $\int_M |u(\hat{t})|^2 dg_{\hat{t}}$.
\end{remark}
The Sobolev space $\mcv$ is closed in $H^1(M^\circ) \times H^1(\partial M)$ and is a complex Hilbert space considering the inner product 
\begin{align}\label{eq:inner_product_v}
\left( 
\begin{pmatrix}
a_1 \\
a_2
\end{pmatrix}
,
\begin{pmatrix}
b_1 \\
b_2
\end{pmatrix}
\right)_\mcv = \int_M   g(\nabla a_1, \nabla \ol{b_1}) dg + \int_{\partial M}   h(\nabla a_2, \nabla \ol{b_2}) dh \\ - \intm m a_1 \ol{b_1} dg - \intb m_b a_2 \ol{b_2} dh
\end{align}
Then $\mx$ is also a Sobolev space with its natural inner product.
\begin{remark}\label{remark:equivalent_norms}
The assumption \eqref{hyp:mass} entails that the norm induced in $\mx$ by the inner products of $\mathcal{U}$ and $\mathcal{V}$ is equivalent to the usual norm of $H^1(M^\circ) \times H^1(\partial M)$. However, the last two terms, missing in the inner product chosen in \cite{Vitillaro16,Vanspranghe20,Vanspranghe22}, play a fundamental role to deal with the terms $mu$ and $m_b u$ in \Cref{eq:non_linear_problem} in the proof of the well-posedness.
\end{remark}
At last we introduce
\begin{equation}
\mcw = \left( H^2(M^\circ) \times H^2(\partial M) \right) \cap \mcv
\end{equation}
which is a dense subspace of $L^2(M^\circ) \times L^2(\partial M)$ \cite{Vitillaro16,Vanspranghe20,Vanspranghe22}.
Setting $v = \gamma_- u$, $w =  -   \partial_t u$ and $z = -   \partial_t v$, the problem stated in \Cref{eq:linear_problem} can be recasted as follows.
\begin{equation}\label{eq:evolution_linear}
\dfrac{d}{dt}
\begin{pmatrix}
u \\ v \\ w \\ z
\end{pmatrix}
+ A(t)
\begin{pmatrix}
u \\ v \\ w \\ z
\end{pmatrix}
=
\begin{pmatrix}
0 \\ 0 \\ F \\ G
\end{pmatrix}
\end{equation}
where $A(t):\mcw \times \mcv \rightarrow \mx$ is defined by
\begin{equation}\label{eq:operator_A}
\begin{pmatrix}
0 & 0 & -1 & 0 \\
0 & 0 & 0 & -1 \\
-\lb-m & 0 & 0 & 0\\
\gamma_+ & -\lb_b -m_b & 0 & 0 
\end{pmatrix}
\end{equation}
with $\lb,m,\lb_b,m_b$ depending on time. We stress that, as a consequence of the fact that we are working on a compact manifold with a smooth metric, $\mathcal{D}(A)$ does not depend on time. 
\section{The linear autonomous equation}\label{sec:time_independent_metric}
The first step to prove the well-posedness of the evolution problem is to study the operator $A(t)$ at a fixed time $t$, a case which encompasses the situation where a fixed time-independent metric $g_t = g \quad \forall t \in [0,T]$ is considered. In this is the case the operator $A$ does not depend on time and in this section we simply write $A$ instead of $A(t)$. Our analysis will also yield the well-posedness of the time-independent massive linear problem, treating the massive terms without resorting to perturbation theory. To prove that $A$ is the generator of a $\mathcal{C}^0$ semigroup, we follow \cite{Liu07,Vanspranghe22}, adapting the proof to the case at hand. First, we show that the operator $ A+\lambda\mathbb{I}:\mathcal{D}(A) \rightarrow \mx$ is surjective. 
\begin{lemma}\label{lemma:A_minus_lambda_surjective}
Let $A$ be the operator defined in \Cref{eq:operator_A}, with $m<0$ and $m_b<0$. Then $\forall \lambda > 0$ the operator $ A+\lambda\mathbb{I}:\mathcal{D}(A) \rightarrow \mx$ is surjective.
\end{lemma}
\begin{proof}
Let $\lambda>0$ and let $\mathcal{F} = (F_1,F_2,F_3,F_4) \in \mathcal{V} \times \mathcal{U}$. We need to find $X= (u,v,w,z) \in \mathcal{D}(A)$ such that 
\begin{equation}\label{eq:surj}
AX + \lambda X = \mathcal{F}
\end{equation}
which is tantamount to
\begin{subequations}
\begin{align}
& -w + \lambda u = F_1 \ \label{eq:surj1} \\
& -z+ \lambda v = F_2 \ \label{eq:surj2}  \\
& - \lb u - m u + \lambda w = F_3 \ & \textit{in} \ M^\circ \ \label{eq:surj3} \\
& - \lb_b v - m v + \gamma_+ u +\lambda z = F_4 \ & \textit{on} \ \partial M \ \ \label{eq:surj4} 
\end{align}
\end{subequations}
\Cref{eq:surj3,eq:surj4} yield that a solution $X$ to \Cref{eq:surj} must satisfy the following variational problem:
\begin{equation}\label{eq:surj_variational}
\begin{split}
\intm g(\nabla u, \nabla \ol{a}) dg + \intb h(\nabla v, \nabla \ol{b}) dh -  \intm m u \ol{a} \ dg - \intb m_b v \ol{b} \ dh \\ + \lambda \intm w \ol{a} \ dg + \lambda \intb z \ol{b} \ dh = \intm F_3 \ol{a} \ dg + \intb F_4 \ol{b} \ dh
\end{split}
\end{equation}
for every $(a,b) \in \mathcal{V}$.
Solving \Cref{eq:surj1,eq:surj2} with respect to $u$ and $v$ respectively, we can write the previous expression as
\begin{align*}
\dfrac{1}{\lambda} \intm g(\nabla w, \nabla \ol{a}) dg + \dfrac{1}{\lambda} \intb   h(\nabla z, \nabla \ol{b}) dh \\
- \dfrac{1}{\lambda}  \intm m^{-1}  w \ol{a} \ dg - \dfrac{1}{\lambda}  \intb m_b  ^{-1} z \ol{b} \ dh + \lambda \intm w \ol{a} \ dg + \lambda \intb z \ol{b} \ dh \\
= \intm F_3 \ol{a} \ dg + \intb F_4 \ol{b} + \dfrac{1}{\lambda} \intm g(\nabla F_1, \nabla \ol{a})  \ dg + \dfrac{1}{\lambda} \intb h(\nabla F_2, \nabla \ol{b})  \ dh \\
+ \dfrac{1}{\lambda} \intm m F_1 \ol{a} \ dg + \dfrac{1}{\lambda} \intb m_b F_2 \ol{b} \ dh
\end{align*}
The bilinear form appearing at the left hand side can be written as:
\begin{equation}\label{eq:bilinear_form}
\begin{split}
\mathcal{B} \left(
\begin{pmatrix}
w\\
z
\end{pmatrix}
,
\begin{pmatrix}
a\\
b
\end{pmatrix}
\right)
=
\dfrac{1}{\lambda} \intm g(\nabla w, \nabla \ol{a}) dg +  \dfrac{1}{\lambda} \intb h(\nabla z, \nabla \ol{b}) dh \\
+ \dfrac{1}{\lambda} \intm \left( -m^{-1} + \lambda^2)  \right) w \ol{a} \ dg + \dfrac{1}{\lambda} \intb \left( -m_b  ^{-1} + \lambda^2   \right) z \ol{b} \ dh
\end{split}
\end{equation}
We claim that $\mathcal{B}$ is continuous and coercive for every $\lambda>0$. Indeed, applying the Cauchy-Schwartz inequality and considering the \Cref{hyp:mass} on $m,m_b$, we readily find that
\begin{equation}
\mathcal{B} \left(
\begin{pmatrix}
w\\
z
\end{pmatrix}
,
\begin{pmatrix}
a\\
b
\end{pmatrix}
\right) \leq C_\lambda 
\left\|
\begin{pmatrix}
w\\
z
\end{pmatrix}
\right\|_{\mathcal{V}}
\cdot
\left\| 
\begin{pmatrix}
a\\
b
\end{pmatrix}
\right\|_{\mathcal{V}}
\end{equation}
with $C_\lambda =  \max \left\{ \lambda + \sup_{p \in M^\circ} |m^{-1}(p)|,\lambda + \sup_{q \in \partial M} |m_b^{-1}(q)| \right\} $ a positive constant.
As for the coercivity of the form, a straightforward computation yields 
\begin{equation}
\mathcal{B} \left(
\begin{pmatrix}
w\\
z
\end{pmatrix}
,
\begin{pmatrix}
w\\
z
\end{pmatrix}
\right)
 \geq K_\lambda \left\|
\begin{pmatrix}
w\\
z
\end{pmatrix}
\right\|_{\mathcal{V}}^2
\end{equation}
with $K_\lambda = \dfrac{1}{\lambda}$. By Lax-Milgram theorem we conclude that the variational problem \Cref{eq:surj_variational} admits a unique solution $(w,z)$ in $\mathcal{V}$, which in turns determines uniquely $u$ and $v$ by means of \Cref{eq:surj1,eq:surj2}. Now it remains to prove that the solution $(u,v,w,z)$ is in $\mathcal{D}(A)$. Let $a$ be a test function in $\mathcal{C}_0^\infty(M^\circ)$. Since $\intm \lb u a \ dg = - \intm g(\nabla u,\nabla a) \ dg $, with $\nabla u \in L^2(M^\circ)$, from \Cref{eq:surj_variational} we conclude that \Cref{eq:surj3} is satisfied almost everywhere in $M^\circ$ and that therefore $\lb u$, a priori defined in $H^{-1}(M^\circ)$, is indeed in $L^2(M^\circ)$.

To improve the regularity of $v$, we write the variational problem in \Cref{eq:surj_variational} in terms of $v,z$ and $\gamma_+ u$. After an integration by parts and using \Cref{eq:surj3} we find that a weak solution $v$ must satisfy
\begin{equation*}
-\intb \lb_b v \ol{b} dh  -\intb \gamma_+ u \gamma_-\ol{a} dh - \intb m_b v \ol{b} dh + \lambda \intb z \ol{b} = \intb F_4 \ol{b} dh
\end{equation*}
for every $(a,b) \in \mathcal{V}$. Since $(a,b) \in \mathcal{V}$, making use of the pairing between $H^{1/2}(\partial M)$  and $H^{-1/2}(\partial M)$, we conclude that $\lb_b v$ -- a priori a distribution in $H^{-1}(\partial M)$ -- is indeed in $H^{-1/2}(\partial M)$ and that the equation
\begin{equation*}
- \lb_b v - \gamma_+ u - m_b v + \lambda z = F_4
\end{equation*}
is satisfied almost everywhere, where the distributional action of $\gamma_+ u$ is defined by 
\begin{equation*}
(\gamma_+ u,\zeta)_{L^2(\partial M)} = \int_M g(\nabla u, \nabla \xi) dg - \int_M \lb u \overline{\xi} dg
\end{equation*}
for every $\zeta \in H^{1/2}(\partial M)$ such that there is $\xi \in H^1(M)$ with $\zeta = \gamma_- \xi$. Then elliptic regularity theory yields that $v \in H^{3/2}(\partial M)$ with
\begin{equation*}
\| v \|_{H^{3/2}(\partial M)} \leq C \left( \|\gamma_+ u \|_{H^{-1/2}(\partial M)} + \|g_2 -\lambda v\|_{L^2(\partial M)} \right).
\end{equation*}
Since $\lb u \in L^2(M^\circ)$, the previous expression yields
\begin{equation*}
\| v \|_{H^{3/2}(\partial M)} \leq C \left( \| \lb u \|_{L^{2}(M^\circ)} + \| u \|_{H^1(M^\circ)} + \|g_2 -\lambda v\|_{L^2(\partial M)} \right).
\end{equation*}
At last, making use of \Cref{eq:surj3}, which is satisfied almost everywhere, and taking into account that $m$ is bounded, we arrive to
\begin{equation}\label{eq:v_h_3/2}
\| v \|_{H^{3/2}(\partial M)} \leq C \left( \| F_3 - \lambda u  \|_{L^{2}(M)} + \| u \|_{H^1(M)} + \|g_2 -\lambda v\|_{L^2(\partial M)} \right)
\end{equation}
with $C$ a real positive constant. Since $\lb u \in L^2(M^\circ)$ and $v = \gamma_- u \in H^{3/2}(\partial M)$, we conclude that $u \in H^2(M)$. To improve the regularity of $v$ once more, we observe that \Cref{eq:v_h_3/2} yields that \Cref{eq:surj4} is satisfied almost everywhere and therefore elliptic regularity theory \cite{Evans10,Taylor11} yields that $v \in H^2(\partial M)$. At last, from \Cref{eq:surj1,eq:surj2} we get that $(w,z) \in  \mathcal{V}$, and therefore $(u,v,w,z) \in \mathcal{D}(A)$, which proves the thesis.
\end{proof}
\begin{remark}\label{rem:massless_bilinear_form}
In the bilinear form given in \Cref{eq:bilinear_form} the masses appears as inverse powers, so we cannot consider $m=0$ or $m_b=0$ directly. However we can include these cases making use of the correct bilinear form. Consider for example the case $m=m_b=0$ studied in \cite{Vanspranghe22}. In this case the bilinear form associated with the variational problem in \Cref{eq:surj_variational} is:
\begin{equation*}
\begin{split}
\mathcal{B} \left(
\begin{pmatrix}
w\\
z
\end{pmatrix}
,
\begin{pmatrix}
a\\
b
\end{pmatrix}
\right)
=
\dfrac{1}{\lambda} \intm g(\nabla w, \nabla \ol{a}) dg +  \dfrac{1}{\lambda} \intb h(\nabla z, \nabla \ol{b}) dh \\
+ \lambda \intm  w \ol{a} \ dg + \lambda \intb z \ol{b} \ dh
\end{split}
\end{equation*}
Given these ingredients the proof of  \Cref{lemma:A_minus_lambda_surjective} keeps to hold true. The two cases in which just one of the two masses is zero can be treated in a similar fashion, setting them to zero in \Cref{eq:surj_variational} and then considering the corresponding bilinear form.
\end{remark}
Now we show that the operator $A$ as per \Cref{eq:operator_A} is dissipative.
\begin{lemma}\label{lemma:A_dissipative}
The operator $A$ defined in \Cref{eq:operator_A} is dissipative, with $Re(X,AX)_\mx=0$ for every $X \in \mathcal{D}(A)$.
\end{lemma}
\begin{proof}
Let $X=(u,v,w,z) \in \mathcal{D}(A)$.
After an integration by parts we find that
\begin{equation}
\begin{split}
(X,AX)_\mx = 
- \intm g(\nabla u, \nabla \ol{w}) dg - \intb g(\nabla v, \nabla \ol{z}) dh \\
+ \intm g(\nabla \ol{u}, \nabla w) dg + \intb g(\nabla \ol{v}, \nabla z) dh\\
+ \intm m u \ol{w} dg - \intm m w \ol{u} dg + \intb m_b v \ol{z} dh - \intb m_b z \ol{v} dh \\
= 2i Im \left(
\begin{pmatrix}
u \\
v
\end{pmatrix}
,
\begin{pmatrix}
w \\
z
\end{pmatrix}
\right)_V
\end{split}
\end{equation}
from which we find that $Re(X,AX)_\mx=0$.
\end{proof}
\begin{remark}
Without the last two terms in the inner product of $\mathcal{V}$ as per \Cref{eq:inner_product_v} the operator $A$ would not have been dissipative.
\end{remark}
\begin{proposition}\label{prop:strongly_continous}
The operator $A$ generates a strongly continuous semigroup $S_t = e^{At}$ of linear contractions on $\mx$.
\end{proposition}
\begin{proof}
We know from \Cref{lemma:A_minus_lambda_surjective,lemma:A_dissipative} that $A$ is a maximally dissipative operator, therefore the thesis follows applying Lumer-Phillips theorem.
\end{proof}
\begin{proposition}\label{prop:evolution_well_posedness}
Let $X(0) \in \mathcal{D}(A)$ and suppose that the source $\mathcal{F}=(F_1,F_2,F_3,F_4)$ is in $\mathcal{C}^1([0,T],\mx)$. Then the initial value problem in \Cref{eq:evolution_linear} admits a unique classical solution
$X(t) \in C^1([0,T],\mathcal{D}(A))$. In the case in which the initial datum $X(0) \in \mx$, the problem in \Cref{eq:evolution_linear} admits a unique mild solution $X(t) \in C^0([0,T],\mx)$ satisfying
\begin{equation}
\|X(t)\|_\mx \leq   \|X(0)\|_\mx + t \max_{s \in [0,T]} \| \mathcal{F}(s) \|_{\mx}
\end{equation}
\end{proposition}
\begin{proof}
The well-posedness of the problem in \Cref{eq:evolution_linear} and the bound on the solution $X(t)$ are immediate consequences of \Cref{prop:strongly_continous,prop:inhomogeneous_evolution_well_posedness_1,prop:inhomogeneous_evolution_well_posedness_2}. 
\end{proof}
\begin{remark}
This well-posedness result keeps to hold true also if the same problem is considered in $[-T,T]$, due to the symmetry of \Cref{eq:linear_problem} under the change of coordinates $t \mapsto -t$.
\end{remark}
\begin{remark}
Consider the homogeneous problem, with the source terms $F$ and $G$ identically zero. Define the energy $\mathcal{E}(t) = \| X(t) \|_\mx^2$ and consider a classical solution $X(t) \in C^1([0,T],\mathcal{D}(A)) $. A short computation shows that $\partial_t \mathcal{E}(t) = 2Re(X,AX)_\mx$, therefore \Cref{lemma:A_dissipative}, in particular the fact that $Re(X,AX)_\mx=0$, immediately entails the conservation of the total energy of the coupled system, energy which however can be exchanged between the field in interior and the one on the boundary.
\end{remark}
\begin{remark}
The energy estimate in \Cref{prop:evolution_well_posedness} immediately yields the finite speed of propagation of the solutions.
\end{remark}

\section{The nonlinear non-autonomous equation}\label{sec:time_dependent_metric}
In this section we consider the non-autonomous nonlinear problem as per \Cref{eq:non_linear_problem}, considering a time-dependent Riemannian metric $g_t$ satisfying Assumptions \eqref{hypothesis:smooth_metric}.

\subsection{Well-posedness results for the linear non-autonomous problem}
We begin proving a low regularity well-posedness result for the linearized problem, adapting a method  originally developed in \cite{Kato53}. 
\begin{proposition}\label{prop:low_regularity_non_auto_well_posedness}
Let $F \in \mathcal{C}^1([0,T],\md (A))$ and suppose that $X_0 \in \md (A)$. Then the linear non-autonomous problem
\begin{equation*}
\begin{cases}
\dot{X} = A(t) X(t) + F(t), \quad t \in [0,T]\\
X(0) = X_0
\end{cases}
\end{equation*}
admits a unique classical solution $X(t) \in \mathcal{C}^1([0,T],\mx) \cap \mathcal{C}^0([0,T],\mathcal{D}(A))$. Furthermore, in the case in which $X_0 \in \mx$, the problem admits a unique mild solution $X \in \mathcal{C}^0([0,T],\mx)$ such that
\begin{equation}\label{eq:bound_low_regularity_non_auto}
\| X(t) \|_\mx \leq \|  X(0) \|_\mx + t \max_{s \in [0,T]} \| F(s) \|_\mx. 
\end{equation}
\end{proposition}
\begin{proof}
We prove our claim making use of \Cref{prop:well_posedness_hyperbolic_problems}.
Condition C1 is true by  \Cref{prop:evolution_well_posedness}. To check the validity of the other assumptions, let us write the operator 
\begin{equation*}
B(t,s) = \left( A(t) + \mathbb{I} \right) \left(A(s) + \mathbb{I} \right)^{-1}
\end{equation*}
defined in \Cref{sec:Non-autonomous_evolution_problems} as follows:
\begin{equation}\label{eq:operator_b_decomposition}
\begin{split}
&
\begin{pmatrix}
0 & 0 & 0 & 0\\
0 & 0 & 0 & 0\\
-\lb (t) - m(t) & 0 & 0 & 0 \\
0 & -\lb_b (t) - m_b(t) & 0 & 0
\end{pmatrix}
\left(A(s) + \mathbb{I} \right)^{-1} \\
+ & \begin{pmatrix}
0 & 0 & -1 & 0\\
0 & 0 & 0 & -1\\
0 & 0 & 0 & 0 \\
\gamma_+ & 0 & 0 & 0 
\end{pmatrix}
\left(A(s) + \mathbb{I} \right)^{-1}
\end{split}
\end{equation}
The second addend, which does not depend on $t$, satisfies conditions C2, C3, C4 and C5 because the trace $\gamma_+$ is bounded on $D(A(s))$ and $ \left(A(s) + \mathbb{I} \right)^{-1}:\mx \rightarrow D(A(s))$ is bounded by for every $s \in [0,T]$ by Hille-Yoshida theorem. Now we focus on the first term. The operator
$$
\begin{pmatrix}
0 & 0 & 0 & 0\\
0 & 0 & 0 & 0\\
-\lb (t) - m(t) & 0 & 0 & 0 \\
0 & -\lb_b (t) - m_b(t) & 0 & 0
\end{pmatrix}
$$
depends on $t$ in a differentiable way by \Cref{hypothesis:smooth_metric,hyp:mass}, which along with the boundedness of  $\left(A(s) + \mathbb{I} \right)^{-1}$ for every $s \in [0,T]$ entails that C4 and C5 are met. For the same reason the first term of the sum in \Cref{eq:operator_b_decomposition} is uniformly bounded (C2) and of bounded variation (C3) for $s,t \in [0,T]$. The hypotheses of \Cref{prop:well_posedness_hyperbolic_problems} are then met and the thesis follows, the bound on $\| X(t) \|_\mx$ being an immediate consequence of Duhamel's principle.
\end{proof}
\subsection{The nonlinear non-autonomous problem}
Finally we focus on the nonlinear problem as per \Cref{eq:non_linear_problem}, which we study in the form
\begin{equation}\label{eq:evolution_nonlinear}
\dfrac{d}{dt}
\begin{pmatrix}
u \\ v \\ w \\ z
\end{pmatrix}
+ A(t)
\begin{pmatrix}
u \\ v \\ w \\ z
\end{pmatrix}
=
\begin{pmatrix}
0 \\
0 \\
\mathcal{N} \left(t,p_1,u(t)(p_1)\right) \\ \mathcal{N}_b \left(t,p_2,v(t)(p_2)\right)
\end{pmatrix}
\end{equation}
where the linear operator $A(t)$ was defined in \Cref{eq:operator_A}.
For every $t \in \mathbb{R}$, $p=(p_1,p_2) \in M \times \partial M$ and $X=(u,v,w,z) \in \mx$ we set
\begin{equation}
\mathcal{G}(t,p,X)=
\begin{pmatrix}
0 \\
0 \\
\mathcal{N} \left(t,p_1,u(t)(p_1)\right)\\
\mathcal{N}_b \left(t,p_2,v(t)(p_2)\right)
\end{pmatrix}
\end{equation}
and
\begin{equation}
\mathcal{F}(t,X)(p)=\mathcal{G}(t,p,X)
\end{equation}
For every $X \in \mx$, by \Cref{hyp:nonlinearities_1,hyp:nonlinearities_2} the function $t \mapsto \mathcal{F}(t,X)$ is mapping continuously $[0,T]$ to $\mx$, therefore for every $T>0$ we have $\mathcal{F}(t,X) \in \mathcal{C}^0([0,T],\mx)$.
Now \Cref{eq:evolution_nonlinear} can be written as
\begin{equation}\label{eq:evolution_nonlinear_2}
\dfrac{d}{dt}X + A(t)X = \mathcal{F}(t,X)
\end{equation}
For $\tau,\rho > 0 $ we introduce the space
$$B_{\tau,\rho} = \left\{ f \in \mathcal{C}^0\left( [0,\tau] , \mx \right) \ \textit{such that} \ \|f\|_\infty \leq \rho  \right\}$$
where $\| \cdot \|_\infty$ is the supremum norm on the space $\mathcal{C}^0\left( [0,\tau],\mx \right)$. A standard argument yield that the pair $\left( B_{\tau,\rho},\| \cdot  \|_\infty \right)$ is a complete metric space. In this space the nonlinear term $\mathcal{F}$  satisfies the following lemma.
\begin{lemma}\label{lemma:nonlinear_lipschitz}
For every $\tau,\rho>0$ there is $L_{\tau,\rho}>0$ such that 
\begin{equation*}
\| \mathcal{F}(t,X)-\mathcal{F}(t,Y) \|_{\infty} \leq L_{\tau,\rho} \|X-Y\|_\infty
\end{equation*}
for every $X,Y \in B_{\tau,\rho}$.
\end{lemma}
\begin{proof}
Let $X=(u_X,v_X,w_X,z_X) \in B_{\tau,\rho}$. Then Assumptions \eqref{hyp:nonlinearities_1} and \eqref{hyp:nonlinearities_2}, together with Sobolev embedding theorem, yield that 
$\mathcal{N}(t,p,u_X) \in L^2(M^\circ)$ and $\mathcal{N}_b(t,q,v_X) \in L^2(\partial M)$ for every $t \in [0,\tau], p \in M, q \in \partial M$, where $\mathsf{C_{M^\circ}}$ and $\mathsf{C_{\partial M}}$ are the Sobolev embedding constants given in \eqref{eq:critical_exponents}. Let $Y=(u_Y,v_Y,w_Y,z_Y) \in B_{\tau,\rho}$. The difference
\begin{equation*}
\mathcal{F}(t,X)-\mathcal{F}(t,Y) =
\begin{pmatrix}
0\\
0\\
\mathcal{N}(t,p,u_X) - \mathcal{N}(t,p,u_Y)  \\
\mathcal{N}_b(t,p,v_X) - \mathcal{N}_b(t,p,v_Y)  \\
\end{pmatrix}.
\end{equation*}
is in $\mx$ because $\mathsf{C_{M^\circ}}$ and $\mathsf{C_{\partial M}}$ are both greater than 2. Using \Cref{hyp:nonlinearities_1,hyp:nonlinearities_2},  H\"{o}lder's inequality and Rellich-Kondrachov embedding theorem, we find that
\begin{equation}\label{eq:lipschitz_h1_interior}
\| \mathcal{N}(t,p,u_X) - \mathcal{N}(t,p,u_Y)  \|_{L^2(M^\circ)} \leq C(\rho,t) \| u_X(t) - u_Y(t) \|_{H^1(M)}
\end{equation}
and
\begin{equation}\label{eq:lipschitz_h1_boundary}
\begin{split}
\| \mathcal{N}_b(t,p,v_X) - \mathcal{N}_b(t,p,v_Y)  \|_{L^2(\partial M)} \leq C_b(\rho,t) \| v_X(t) - v_Y(t) \|_{H^1(\partial M)}
\end{split}
\end{equation}
for every $t \in [0,\tau]$, where $C(\rho,\tau)$ and $C_b(\rho,\tau)$ are non-decreasing in $\tau,\rho$ by  \Cref{hyp:nonlinearities_1}.
Let us prove this claim for \Cref{eq:lipschitz_h1_interior} -- The proof of \Cref{eq:lipschitz_h1_boundary} being analogous.
By H\"{o}lder's inequality and \Cref{hyp:nonlinearities_1} we find
\begin{equation*}
\begin{split}
\| \mathcal{N}(t,p,u_X) - \mathcal{N}(t,p,u_Y)  \|_{L^2(M^\circ)} \leq \\
C_{\tau,\rho} \| u_x - u_Y \|_{L^{2a}(M)}^{1/2a}  \| 1 + |u_X|^{\alpha-1} + |u_Y|^{\alpha-1} \|_{L^{2b}(M)}^{1/2b} 
\end{split}
\end{equation*}
with $a,b$ conjugate exponents and $C_{\tau,\rho}$ the constant appearing in \Cref{eq:non_linearities_form}. Thanks to \Cref{hyp:nonlinearities_2}, choosing $2b<\mathsf{C_{M^\circ}}/(\alpha-1)$ such that the conjugate exponent $2a$ is less than $\mathsf{C_{M^\circ}}$ and applying Rellich-Kondrachov embedding theorem, we get our claim with
\begin{equation*}
C(\tau,\rho) =  C_{\tau,\rho} \| 1 + |u_X|^{\alpha-1} + |u_Y|^{\alpha-1} \|_{L^{2b}(M)}^{1/2b} .
\end{equation*}
This choice of the exponents $a,b$ can be always done. The inequality $2a < \mathsf{C_{M^\circ}}$ is tantamount to require $b > \mathsf{C_{M^\circ}}/(\mathsf{C_{M^\circ}} - 2)$. Combining this request with $2b<\mathsf{C_{M^\circ}}/(\alpha-1)$ we get the following inequality for $\alpha$:
\begin{equation*}
\dfrac{\mathsf{C_{M^\circ}}}{\mathsf{C_{M^\circ}}-2} \leq \dfrac{\mathsf{C_{M^\circ}}}{2(\alpha-1)}
\end{equation*}
which yields $\alpha \leq \mathsf{C_{M^\circ}}/2$ -- satisfied by \Cref{hyp:nonlinearities_2}.
Collecting the above inequalities concerning the non-linear terms we find 
\begin{equation*}
\| \mathcal{F}(t,X)-\mathcal{F}(t,Y) \|_\mx \leq L_{\tau,\rho} \|X(t)-Y(t)\|_\mx
\end{equation*}
with $L_{\tau,\rho} = \max (C(\rho,\tau),C_b(\rho,\tau))$.
The thesis follows taking the sup in $t \in [0, \tau]$ of the above inequality.
\end{proof}

\begin{remark}
The lemma above holds true also adding linear terms in $u_t$ and $v_t$, for example considering a non-linearity of the form
\begin{equation*}
\begin{split}
\mathcal{N}(t,x,y,z) & = |y|^{\alpha-1}y + Q(t,x) z\\
\mathcal{N}_b(t,x,y,z) & = |y|^{\beta-1}y  + Q_b(t,x) z
\end{split}
\end{equation*}
with $Q \in \mathcal{C}^\infty([0,T]\times M)$ and $Q_b \in \mathcal{C}^\infty([0,T]\times \partial M)$. Indeed the bound 
\begin{equation*}
\| \mathcal{F}(t,X)-\mathcal{F}(t,Y) \|_\mx \leq L_{\tau,\rho} \|X(t)-Y(t)\|_\mx
\end{equation*}
given in the previous proof keep to hold true even adding these terms.
\end{remark}
Now we can prove the following well-posedness result making use of a fixed point argument.
\begin{proposition}\label{prop:nonlinear_well_posedness_conditional}
Assume that Hypotheses \eqref{hypothesis:smooth_metric}, \eqref{hyp:nonlinearities_1} and \eqref{hyp:nonlinearities_2} hold true. Let $\rho > 0$. Then there is $\tau = \tau(\rho)$ such that for every $X_0 \in \mx$ with $\| X_0 \|_\mx \leq \rho$, the problem in \Cref{eq:evolution_nonlinear_2} admits a unique mild solution $X \in C^0([0,\tau],\mx)$ such that $\| X(t) \|_\infty \leq  \rho (M_0 + 1)$, where $M_0 = \sup_{t,s \in [0,T]} \| U(s,t) \|$.
\end{proposition}
\begin{proof}
Given $\rho > 0$ consider an initial datum $X_0$ such that $\|X_0\|_\mx \leq \rho$. Set $r = 1 + M_0 \rho$ and let $0 < \tau \leq T$. Consider the map
\begin{equation}\label{eq:fixed_point_map}
\Phi_{X_0}(X)(t) = U(t,0) X_0 + \int_0^t U(t,s)F(s,X(s))ds
\end{equation}
where $X \in B_{\tau,r}$. We claim that $\Phi$ is a continuous function in $t$ by the continuity of $F$ and the dominated convergence theorem. The continuity of $F$ is a consequence of \Cref{hyp:nonlinearities_1,hyp:nonlinearities_2}. Then we need to prove that for every $0 \leq t \leq T$ 
$$\lim_{t \rightarrow t_0} \Phi_{X_0}(X)(t) = \Phi_{X_0}(X)(t_0)$$
As for the first term in \Cref{eq:fixed_point_map} we note that by definition of evolution family $\lim_{t \rightarrow t_0} U(t,0) X_0 = U(T_0,0)X_0$. For the second term we find that
\begin{equation}
\begin{split}
\left| \int_0^t U(t,s) F(X(t)) ds \right| \leq M_0 \int_0^t \left| F(s,X(t)) \right| \leq M \tau \|F\|_\infty
\end{split}
\end{equation}
and our claim is proved. Therefore we conclude that the map $\Phi_{X_0}(X)$ is in $\mathcal{C}^0([0,\tau],\mx)$ and that $X \in B_{\tau,\rho}$ is a solution of \Cref{eq:evolution_nonlinear_2} with initial datum $X_0$ if and only if $X$ is a fixed point of $\Phi_{X_0}$. For every $X,Y \in B_{\tau,\rho}$ we have
\begin{equation}
\begin{split}
\left\|\Phi_{X_0}(X)(t) - \Phi_{X_0}(Y)(t) \right\|_\mx & \leq M_0 \int_0^t \left\| \mathcal{F}(s,X(s))- \mathcal{F}(s,Y(s)) \right\|_\mx ds \\
& \leq M_0 \tau \left\| \mathcal{F}(s,X(s))- \mathcal{F}(s,Y(s)) \right\|_\mx
\end{split}
\end{equation}
from which, making use of \Cref{lemma:nonlinear_lipschitz} we find
\begin{equation}\label{eq:evolution_lipschitz}
\left\|\Phi_{X_0}(X)(t) - \Phi_{X_0}(Y)(t) \right\|_\infty \leq \tau M_0 L_{\tau,\rho} \left\| X-Y \right\|_\infty
\end{equation}
Since $L_{\tau,\rho}$ is non-decreasing in $\tau$, for $\tau$ small enough we have $\tau M_0 L_{\tau,\rho} < 1$ and the map $\Phi_{X_0}$ is a contraction on $B_{\tau,\rho}$. Thus by Banach's theorem it admits a unique fixed point $X \in B_{\tau,\rho}$, which is the only solution to \Cref{eq:evolution_nonlinear_2} with initial datum $X_0$.  For $t \in [0,\tau]$ we have the following bound on the $\mx$ norm of $X(t)$:
\begin{equation}\label{eq:nonlinear_solution_bound}
\begin{split}
\|X(t)\|_\mx = \| \Phi_{X_0}(X)(t) \|_\mx \leq M_0 \| X_0 \|_\mx +  M_0 \int_0^t \|\mathcal{F}(s,X(s)) - \mathcal{F}(s,0) \|_\mx ds\\ \leq M_0 \rho + M_0 \tau L_{\tau,\rho} \rho
\end{split}
\end{equation}
where we used that $\mathcal{F}(s,0)=0$ and \Cref{lemma:nonlinear_lipschitz}.
Choose now 
\begin{equation}\label{eq:existence_time}
\tau = \min \left\{T, \dfrac{1}{2 M_0 L_{T,\rho}} \right\}
\end{equation}
 yielding a Lipschitz constant of $1/2$. Then we can bound the norm of $X(t)$ independently of $\tau$ as
\begin{equation*}
\| X(t) \|_\mx \leq \rho \left( 1 + M_0 \right)
\end{equation*}
from which the thesis follows.
\end{proof}
\begin{remark}\label{remark:shifted_evolution}
Notice that the previous result keeps to hold true also using a shifted evolution family $U(t+\tau,s+\tau)$, with $s\leq t$ and $t \in [0,T-\tau]$, namely substituting the map defined in \Cref{eq:fixed_point_map} with
\begin{equation*}
\Phi_{X_0}(X)(t) = U(t+\tau,\tau) X_0 + \int_0^t U(t+\tau,s+\tau)F(s,X(s))ds.
\end{equation*}
This formulation is useful when the initial datum is assigned at time $t=\tau$ instead of $t=0$.
\end{remark}

This local well-posedness proposition is a conditional uniqueness result among functions belonging to a certain ball. To derive unconditional uniqueness, we need to be able to glue and shift solutions in time.
\begin{lemma}\label{lemma:glue_and_shift}
Consider $0 < \tau_1 < \tau_2 < T$. Let $X \in \mathcal{C}^0([0,\tau_1],\mx)$ be a mild solution in $[0,\tau_1]$ of the problem in \Cref{eq:evolution_nonlinear_2} with initial datum $X_0$. Then:
\begin{itemize}
\item[a)] If $Y \in \mathcal{C}^0([\tau_1,\tau_2],\mx)$ is a mild solution with initial datum $Y(\tau_1)=X(\tau_2)$ then the function $Z \in \mathcal{C}^0([0,\tau_2],\mx)$ given by
\begin{equation*}
Z(t) =
\begin{cases}
X(t) \quad \textit{if} \quad 0 \leq  t \leq \tau_1\\
Y(t) \quad \textit{if} \quad \tau_1 < t \leq \tau_2
\end{cases}
\end{equation*}
is a mild solution of the problem in $[0,\tau_2]$ with initial datum $X_0$. 
\item[b)] Let $\rho \in (0,\tau_1)$. Then the shifted function $$X(\cdot + \rho) \in \mathcal{C}^0([0,\tau_1-\rho],\mx)$$ is a mild solution of the evolution problem with initial datum $X(\rho)$, built with the shifted evolution family $U(\cdot+\rho,\cdot+\rho)$ as per \Cref{remark:shifted_evolution}.
\end{itemize}
\end{lemma}
\begin{proof}
a) By construction $Z \in \mathcal{C}^1([0,\tau_2],\mx)$ and solves the evolution problem on $[0,\tau_1]$. We need to show that $Z$ is also a solution on $[\tau_1,\tau_2]$.
We know that
\begin{equation*}
Z(t) = Y(t-\tau_1) = U(t,\tau_1)X(\tau_1)+\int_{\tau_1}^t U(t,s)F(Y(s))ds
\end{equation*}
Writing on $[0,\tau_1]$ the function $Z(t)$ is nothing but $X(t)$, the previous expression becomes
\begin{equation*}
U(t,\tau_1)\left[ U(\tau_1,0)X(0) + \int_0^{\tau_1} U(\tau_1,s) F(X(s)) ds \right] +\int_{\tau_1}^t U(t,s)F(Y(s))ds
\end{equation*}
from which we read that
\begin{equation}
Z(t)=U(t,0)Z(0)+\int_0^t U(t,s) F(Z(t))
\end{equation}
\linebreak
b) Let $Y=X(t+\rho)$ with $t \in [0,\tau_1-\rho]$.
Then by definition of evolution family we have
\begin{equation}\label{eq:shift_first_step}
\begin{split}
Y(t) & =   X(t+\rho) = U(t+\rho,0)X(0)+ \int_0^{t+\rho} U(t+\rho,s)F(X(s)) ds \\
& = U(t+\rho,\rho)U(\rho,0)X(0)+ \int_0^{t+\rho} U(t+\rho,s)F(X(s)) ds.
\end{split}
\end{equation}
Writing the last integral as
\begin{equation*}
\int_0^{\rho} U(t+\rho,\rho)U(\rho,s)F(X(s)) ds + \int_\rho^{t+\rho} U(t+\rho,s)F(X(s)) ds
\end{equation*}
the last expression in \Cref{eq:shift_first_step} becomes
\begin{equation*}
U(t+\rho,\rho) \left[ U(\rho,0)X(0) + \int_0^{\rho} U(\rho,s)F(X(s)) ds  \right] + \int_\rho^{t+\rho} U(t+\rho,s)F(X(s)) ds.
\end{equation*}
Since the term in the square brackets is nothing but $X(\rho)$, we finally get that
\begin{equation*}
\begin{split}
Y(t) & = U(t+\rho,\rho) X(\rho)+\int_\rho^{t+\rho} U(t+\rho,s)F(X(s)) ds\\
& = U(t+\rho,\rho) X(\rho) + \int_0^{t} U(t+\rho,s+\rho) F(Y(s)) ds
\end{split}
\end{equation*}
Now we can prove the following unconditional uniqueness result for all the mild solutions to \Cref{eq:evolution_nonlinear_2}.
\begin{lemma}\label{lemma:uniqueness_unconditional}
Let $X$ and $Y$ be mild solutions of \Cref{eq:evolution_nonlinear_2} with initial datum $X(0)=Y(0)$ on $I_X$ and $I_Y$ respectively, $I_X,I_Y \subseteq [0,T]$. Then $X \equiv Y$ on $I_X \cap I_Y$.
\end{lemma}
We prove the statement by contradiction. Let $I = I_X \cap I_Y$. Since $X(0)=Y(0)$, then we have that 
$$\tau = \sup\{ \kappa \in I \ | \ \forall t \in [0,\kappa] \quad X(t)=Y(t) \} $$
is in $[0,\sup I]$. Suppose that $X \neq Y$ on the interval $I$. Then by continuity we find that $\tau < \sup I$ and $X(\tau)=Y(\tau)$. Therefore there is a sequence of times $\{ t_n \} \subseteq I$ with $t_n \rightarrow \tau^+$ for $n \rightarrow + \infty$ such that for each $n$ we have $X(t_n) \neq Y(t_n)$. Now let us consider $\omega^\star>0$ such that $\tau+\omega \in I$. If $\omega \in [0,\omega^\star]$, \Cref{lemma:glue_and_shift} shows that the shifted solutions $\hat{X}(t)=X(t+\tau)$ and $\hat{Y}(t)=Y(t+\tau)$ are mild solutions on $[0,\omega^\star]$ with initial datum $X(\tau)=Y(\tau)$. By \Cref{prop:nonlinear_well_posedness_conditional} and \Cref{remark:shifted_evolution}, for small times we have $\hat{X}(t) = \hat{Y}(t)$, contradicting the inequality $X(t_n) \neq Y(t_n)$ for $n$ sufficiently large.
\end{proof}
At last we prove the uniqueness of the maximal solution and the continuous dependence on the initial data. Using the lemmas above, the proof is similar to the one given for autonomous equations, see e.g. \cite[Theorem 1.11]{Sch22nl}, therefore we limit ourselves to summarize the main steps.
\begin{proposition}\label{prop:nonlinear_well_posedness_unconditional}
Assume Hypotheses \eqref{hypothesis:smooth_metric}, \eqref{hyp:nonlinearities_1} and \eqref{hyp:nonlinearities_2}. Then the following assertions hold true:
\begin{itemize}
\item[a)] For each $u_0 \in \mx$ there is a maximal mild solution $X(t) \in \mathcal{C}^0(I,\mx)$ with $I$ either $[0,T]$ or $[0,t^+(X_0))$, where $t^+(X_0) \in [\tau,T]$ and $\tau$ is as per \Cref{eq:existence_time}.
\item[b)] If $t^+(X_0) < T$, then $\lim_{t\rightarrow t^+(X_0)^-} \|X(t) \|_\mx =+\infty$.
\item[c)] For any $t^\star \in (0,t^+(X_0))$ there is a radius $\rho=\rho(X_0,t^\star)$ such that the map
$$ \overline{B(X_0,\rho)} \rightarrow \mathcal{C}^0([0,b],\mx), \quad X_0 \mapsto X(t)$$
is Lipschitz continuous.
\end{itemize}
\end{proposition}
\begin{proof}
First we prove claim b) by contradiction. Let $t^+ < T$ and assume that there is a sequence $\{ t_n \}_{n \in \mathbb{N}} \subset [0,t^+(X_0))$ such that $t_n \rightarrow t^+(X_0)$ for $n \rightarrow +\infty$ with $\sup_{n \in \mathbb{N}}\| X(t_n) \|_\mx < \infty$. Choose $m \in \mathbb{N}$ such that $t^+(X_0) < t_m + \tau < T$, with $\tau$ as per \Cref{eq:existence_time}. Then by \Cref{prop:nonlinear_well_posedness_conditional} and \Cref{remark:shifted_evolution} we obtain a mild solution $\hat{X}(t)$ on $[t_m, t_m +\tau]$ with initial datum $X(t_m)$. Gluing this solution together with the starting solution $X(t)$, by \Cref{lemma:glue_and_shift} we obtain a solution on $[0,t_m + \tau]$, in contrast with the fact that $[0,t^+(X_0))$ is the maximal existence interval. As for claim a), the maximal solution is built in the usual way, shifting and gluing together local solutions given by \Cref{prop:nonlinear_well_posedness_conditional}. Applying \Cref{lemma:glue_and_shift,lemma:uniqueness_unconditional} then yields that the solution is uniquely defined on the the maximal existence interval $I$, which by assertion b) is either $[0,T]$ or $[0,t^+(X_0))$. 
Finally, statement c) follows using the bounds in \Cref{eq:nonlinear_solution_bound,eq:evolution_lipschitz} in $[0,t^\star]$.
\end{proof}

\section*{acknowledgements}
We are grateful to Claudio Dappiaggi for the useful discussions. This work is supported by a postdoctoral fellowship of the Universit\`a degli Studi di Pavia, which is gratefully acknowledged.

\newcommand{\etalchar}[1]{$^{#1}$}

\end{document}